\documentclass[11pt]{article}
\usepackage[T1]{fontenc}
\usepackage{lmodern}
\usepackage{amsmath,amssymb,amsthm,mathtools}
\usepackage[margin=1in]{geometry}
\usepackage{microtype}
\usepackage[colorlinks=true,linkcolor=blue,citecolor=blue,urlcolor=blue]{hyperref}

\newtheorem{theorem}{Theorem}[section]
\newtheorem{lemma}[theorem]{Lemma}
\newtheorem{corollary}[theorem]{Corollary}
\theoremstyle{remark}
\newtheorem*{remark*}{Remark}
\newtheorem{conjecture}[theorem]{Conjecture}

\newcommand{\F}{\mathcal F}
\newcommand{\G}{\mathcal G}
\newcommand{\Hh}{\mathcal H}
\newcommand{\M}{\mathcal M}
\newcommand{\Nn}{\mathcal N}
\newcommand{\Tt}{\mathcal T}
\DeclareMathOperator{\cov}{cov}
\newcommand{\bin}[2]{\binom{#1}{#2}}

\title{The Erd\H{o}s--Chv\'atal Simplex Conjecture}
\author{Jing Huang\footnote{E-mail address: {\it
jinghuangscut@gmail.com
}}
}

\date{\small
School of Mathematics and Information Science,
Guangzhou University, Guangzhou 510006, China
}

\begin{document}
\maketitle

\begin{abstract}
We prove the Erd\H{o}s--Chv\'atal simplex conjecture
across its entire parameter range.
If $2\le d<k$  and $(d+1)k\le dv$, then
every $d$-simplex-free family of $k$-subsets of $[v]$ has at most
$\binom{v-1}{k-1}$ members, with equality only for a full star.

\medskip
\textbf{2020 MSC:} Primary 05D05; Secondary 05C65.

\textbf{Keywords:} Extremal combinatorics; Erd\H{o}s--Chv\'atal simplex conjecture; 
Intersection theorem;  Erd\H{o}s--Ko--Rado theorem.
\end{abstract}

\section{Introduction}

Throughout this paper, let $V=[v]=\{1,\ldots,v\}$, and let
$\binom{V}{k}$ denote the family of all $k$-element subsets
of $V$. A family $\mathcal F\subseteq\binom{V}{k}$ is
\emph{intersecting} if every two of its members have nonempty
intersection. The \emph{full star} with centre $\alpha\in V$ is
\[
 \mathcal S_\alpha(V,k)
   =\{A\in\binom{V}{k}:\alpha\in A\}.
\]
When the ground set and the uniformity are understood, we simply
write $\mathcal S_\alpha$. A full star has
$\binom{v-1}{k-1}$ members. The Erd\H{o}s--Ko--Rado (EKR)
theorem is a foundational result in extremal combinatorics:
if $v\ge2k$, every intersecting family in $\binom{V}{k}$ has
at most $\binom{v-1}{k-1}$ members; if $v>2k$, full stars
are the only families attaining this bound \cite{An,EKR,GM}.

The EKR theorem gives the full-star bound when two disjoint
sets are forbidden. This leads to a natural question: does
the same bound hold when the forbidden configuration involves
three or more sets? The first such problem concerns a
\emph{triangle}: three pairwise intersecting sets with an empty
total intersection. A full star contains no triangle.
In 1971, Erd\H{o}s asked for the maximum size of a
triangle-free uniform family and suggested the full-star
bound \cite{Erdos1971}.
This assertion subsequently became known as
Erd\H{o}s's triangle conjecture \cite{MV}.
The problem asks whether the EKR bound survives a restriction
on three-way intersections.
In 1974, Chv\'atal extended the problem to higher dimensions
\cite{Chvatal}.
In modern terminology, the forbidden configuration is a
$d$-simplex. 
For $d\ge1$, a \emph{$d$-simplex}
is a family of distinct sets $A_1,\ldots,A_{d+1}$ satisfying
\[
 \bigcap_{i=1}^{d+1}A_i=\varnothing,
 \qquad
 \bigcap_{i\ne j}A_i\ne\varnothing
 \quad(1\le j\le d+1).
\]
A family is \emph{$d$-simplex-free} if it contains no such
subfamily. Thus a simplex is a minimal collection of sets with
empty total intersection, and Erd\H{o}s's original triangle
problem is precisely the case $d=2$. 
Since all its members share a common element, a full star
contains no simplex of any dimension.
The Erd\H{o}s--Chv\'atal simplex conjecture proposes a broad
generalization of the EKR theorem: in every dimension $d\ge2$,
forbidding a $d$-simplex should determine both the sharp
extremal bound and the complete structure of the families
attaining it, with full stars as the unique extremal families
throughout the following parameter range \cite{Currier,KL}.

\begin{conjecture}\label{conj:EC}
Let $2\le d<k$  and $(d+1)k\le dv$.
If $\mathcal F\subseteq\binom{V}{k}$ is $d$-simplex-free, then
\begin{equation*} 
 |\mathcal F|\le\binom{v-1}{k-1}.
\end{equation*}
Equality holds if and only if $\mathcal F$ is a full star.
\end{conjecture}

The $d=1$ version of the numerical assertion is exactly the
EKR theorem: a $1$-simplex consists of two disjoint nonempty
sets, and the parameter condition becomes $v\ge2k$.
For $d=1$, full-star uniqueness holds when $v>2k$; at
$v=2k$, additional extremal families occur for $k\ge2$.
In this precise sense, Conjecture~\ref{conj:EC} is a
higher-dimensional generalization of the EKR theorem.
For $d\ge2$, its range also includes values $v<2k$, where
pairwise intersection is automatic and the restriction on
higher-order intersections becomes essential.
Early results established the conjecture in several distinct
regimes. Chv\'atal proved the case $k=d+1$ \cite{Chvatal}.
Frankl's intersection theorem \cite{Frankl1976} yields the
numerical bound in the strict range
$
 k>\frac{d-1}{d}v;
$
this consequence is recorded in \cite[p.~170]{Frankl1981}.
Bermond and Frankl obtained further infinite families of
parameter values \cite{BermondFrankl}. Frankl subsequently
proved an asymptotically sharp bound for fixed $k>d$, together
with exact results for sufficiently large $v$ in several cases
\cite{Frankl1981}. Frankl and F\"uredi then established the
exact bound and full-star uniqueness for every fixed $k>d$
and all $v>v_0(k)$ \cite{FranklFuredi1987}.
For triangles, Cs\'ak\'any and Kahn gave a homological proof
of the $3$-uniform case \cite{CsakanyKahn}, and Mubayi and
Verstra\"ete settled the entire case $d=2$, for all $k\ge3$
and $v\ge3k/2$ \cite{MV}.

Further progress required arguments that allow the uniformity
to grow with the ground set. For every fixed $d\ge2$,
Keevash and Mubayi developed a stability approach and proved
the conjecture when $\zeta v<k<v/2-T$, for sufficiently large
$v$, where $\zeta>0$ is fixed and $T$ depends on $d$ and
$\zeta$ \cite{KeevashMubayi}. Keller and Lifshitz developed
a general junta method for forbidden hypergraphs, approximating
large families by families whose membership is determined by
a bounded number of coordinates. Combined with arguments
specific to simplices, their method proved the conjecture for
every admissible $k$ whenever $v>v_0(d)$
\cite{KL}. Thus the ground-set threshold could
be made independent of $k$, although it still depended on $d$.

A related development concerns forbidden configurations with
small union. Frankl and F\"uredi studied three sets with empty
total intersection and a restriction on their union
\cite{FranklFuredi1983}. Keevash and Mubayi  proposed
the stronger \emph{simplex-cluster conjecture}
\cite{KeevashMubayi}: for $d<k$ and $v>(d+1)k/d$, the same
extremal conclusion should hold when only $d$-simplices whose
union has size at most $2k$ are forbidden. These configurations
are called \emph{$d$-simplex-clusters}. Lifshitz proved this
conjecture when $k$ is at least a fixed positive proportion of
$v$ and $v$ is sufficiently large in terms of that proportion
and $d$ \cite{Lifshitz}. Currier then established it for all
$d\ge3$, $d<k$, and $v\ge2k-d+2$
\cite{Currier}. Every simplex-free family is
simplex-cluster-free, so Currier's theorem also proves
Conjecture~\ref{conj:EC} in this range. These results left
the small-ground-set regime as the remaining obstacle to
a proof covering all parameters simultaneously.

In this paper, we prove the Erd\H{o}s--Chv\'atal simplex
conjecture in its full parameter range, including the
classification of all extremal families.
\begin{theorem}\label{thm:main}
Conjecture~\ref{conj:EC} holds.
\end{theorem}
The theorem completes the determination of the maximum size
and the extremal families for every admissible triple $(v,k,d)$,
including the small-ground-set cases left outside the preceding
general results.  
The proof combines a composition lemma for irredundant covers
with local extension-degree estimates. Let
\begin{equation*} 
 r=v-k,\qquad q=d+1.
\end{equation*}
Complementation turns a $d$-simplex into a cover by $q$ sets
of size $r$, each containing a point in no other set of the
cover. Averaging over ordered partitions combines such covers
on disjoint blocks and also determines the equality cases
(Lemma~\ref{lem:composition}). For the remaining parameters,
we apply Currier's auxiliary simplex-cluster and link estimates
\cite{Currier}. At equality, these estimates force each member
to contain a $(k-1)$-subset lying in no other member.
The rigidity consequence of Bollob\'as's theorem recorded in
\cite[Lemma~2.1]{FL} then identifies the star. Together, these
arguments establish the bound and its equality characterization
throughout the remaining range.
The cover-composition method developed here, specifically the use of ordered partition averaging, may also inform structural characterizations for other forbidden-configuration problems in extremal combinatorics.

The rest of the paper is devoted to proving Theorem \ref{thm:main}.
Section~\ref{sec:preliminaries} records the external results
and extension-degree notation used in the proof.
Section~\ref{sec:rigidity} develops the cover-composition
argument and the local estimate, and combines them to complete the proof.

\section{Preliminaries}\label{sec:preliminaries}
We first record three classical intersection theorems.
The EKR theorem provides the local two-member covers and the
equality case in the composition argument; Chv\'atal's theorem
provides the larger local covers; and the triangle theorem
treats $d=2$.
We begin with the EKR theorem \cite{An,EKR,GM}; its formulation
below is also recorded in \cite[Theorem~1]{Currier}.

\begin{lemma}\label{lem:ekr}
Let $V$ be a set of size $v$, let $k\ge1$, and suppose $v\ge2k$.
If $\F\subseteq\binom{V}{k}$ is intersecting, then
$
 |\F|\le\binom{v-1}{k-1}.
$
If $v>2k$, equality holds if and only if
$
\F=\mathcal S_\alpha(V,k) 
   =\{A\in\binom{V}{k}:\alpha\in A\}
$ is the  full star with centre $\alpha$ 
for some $\alpha\in V$.
\end{lemma}

Chv\'atal proved the following exact result for
$(k-1)$-simplex-free $k$-uniform families
\cite[Theorem, p.~355]{Chvatal}.
\begin{lemma}\label{lem:chvatal}
Let $V$ be a set of size $v$. 
Let $k\ge3$ and $v\ge k+2$.
If $\F\subseteq\binom{V}{k}$ is $(k-1)$-simplex-free, 
then $\vert{}\F\vert{}\le\binom{v-1}{k-1},$ with equality 
if and only if $\F$ is a full star.
\end{lemma}

Another fundamental exact case concerns $2$-simplices. When the forbidden configuration is a triangle, the full star is again the unique extremal structure \cite[Theorem~2, the case $d=2$]{MV}.
\begin{lemma}\label{lem:triangle}
Let $V$ be a set of size $v$.
Let $k\ge3$ and $v\ge3k/2$.
If $\F\subseteq\binom{V}{k}$ is  $2$-simplex-free, then
$
 |\F|\le\binom{v-1}{k-1},
$
with equality if and only if $\F$ is a full star.
\end{lemma}

Only the specialization in Lemma~\ref{lem:triangle} is used.
For $d>2$, the cited general theorem forbids every pairwise
intersecting subfamily of $d+1$ members with empty total
intersection, a stronger restriction than being $d$-simplex-free.
We need more detailed information about how a
set can be extended within the family. For
$\F\subseteq\binom{V}{k}$ and $T\subseteq V$, define the
\emph{extension degree}
\begin{equation*}
 \deg_\F(T)=|\{A\in\F:T\subseteq A\}|.
\end{equation*}
A $(k-1)$-subset $T$ is \emph{uniquely extendible} in $\F$ when
$\deg_\F(T)=1$. If $v=|V|$, then
$0\le\deg_\F(T)\le v-k+1$ for every such $T$.
For $A\in\F$ and $i\ge1$, let 
\begin{equation}\label{eq:extension-classes}
 P_i^\F(A)
   =\{\alpha\in A:\deg_\F(A\setminus\{\alpha\})=i\},
 \qquad p_i^\F(A)=|P_i^\F(A)|.
\end{equation}
The sets $P_i^\F(A)$ correspond to  
$\alpha_\F^i(A)$ in \cite[Definition~3]{Currier};
we call them extension classes and write $p_i^\F(A)$
for their cardinalities. 
When $\F$ is fixed, we abbreviate these symbols to $P_i(A)$ and
$p_i(A)$. The sets $P_i(A)$ are pairwise disjoint subsets of $A$.
A \emph{$d$-simplex-cluster} in $\binom{V}{k}$
is a $d$-simplex $\{A_1,\ldots,A_{d+1}\}$ satisfying
$
\left|\bigcup_{j=1}^{d+1} A_j\right|\le 2k;
$
see \cite[Definition~2]{Currier}.
In particular, every $d$-simplex-free family is
$d$-simplex-cluster-free.

The next four lemmas are auxiliary results from \cite{Currier}. 
The first  turns flexibility of extension into a forbidden
simplex-cluster \cite[Lemma~1]{Currier}. 
We will use its contrapositive to mark the members
of a link that may participate in a disjoint pair.

\begin{lemma}\label{lem:cluster-criterion}
Let $V$ be a set of size $v$.
Let $d\ge2$, $d+1\le k$,  $v\ge2k-d$, and let
$\F\subseteq\binom{V}{k}$.
If there exist $A,B\in\F$ such that
\[
 A\cap B\in
 \binom{A\setminus P_1(A)}{d}\setminus\binom{P_2(A)}{d},
\]
then $\F$ contains a $d$-simplex-cluster.
\end{lemma}

The next estimate limits the total contribution from extension
degrees one and two. It requires a lower bound on $|\F|$, but no
forbidden-configuration assumption \cite[Lemma~2]{Currier}.

\begin{lemma}\label{lem:extension-budget}
Let $1\le k<v$ and let $\F\subseteq\binom{V}{k}$ satisfy
$|\F|\ge\binom{v-1}{k-1}$. Then
\[
 \sum_{A\in\F}
 \left(p_1(A)+\frac{v-k-1}{2(v-k)}p_2(A)\right)
 \le\binom{v-1}{k-1}.
\]
\end{lemma}

The subsequent argument will sum over $d$-subsets, whereas the preceding estimate counts individual elements. The following lemma connects these two levels of counting \cite[Lemma~3]{Currier}.
\begin{lemma}\label{lem:marked-subsets}
Let $d\ge2$, $d+1\le k<v$, and let
$\F\subseteq\binom{V}{k}$. Then
\[
 \sum_{A\in\F}
 \left[\binom{k}{d}-\binom{k-p_1(A)}{d}+\binom{p_2(A)}{d}\right]
 \le\binom{k-1}{d-1}
       \sum_{A\in\F}\left(p_1(A)+\frac{p_2(A)}{d}\right).
\]
\end{lemma}

The following inequality bounds its size when both members
of every disjoint pair belong to a specified marked subfamily
\cite[Theorem~3, numerical inequality]{Currier}.
The paragraph preceding Theorem~3 in Currier's revised preprint
\cite{CurrierV2} attributes Lemma~\ref{lem:marked-intersection}
to Borg and to Borg--Leader.  We do not
use the equality classification, which requires $v>2k$.

\begin{lemma}\label{lem:marked-intersection}
Let $V$ have size $v$, let $k\ge1$, and suppose $v\ge2k$.
Let $\F^*\subseteq\F\subseteq\binom{V}{k}$ be such that if $A,B\in\F$ and $A\cap B=\varnothing$, then $A,B\in\F^*$.
Then
\begin{equation*} 
 |\F|\le\binom{v-1}{k-1}+\frac{v-k}{v}|\F^*|.
\end{equation*}
\end{lemma}

\section{Proof of Theorem \ref{thm:main}}
\label{sec:rigidity}
We first transform the forbidden simplex condition into an equivalent 
covering problem by passing to the complementary family.
Let $V$ be a set having size $v$.
 For
$\F\subseteq\binom{V}{k}$, let
\[
 \G=\{V\setminus A:A\in\F\}\subseteq\binom{V}{r},
 \qquad r=v-k.
\]
An \emph{irredundant cover} of $V$ is a family of subsets
whose union is $V$ and in which every member has a
\emph{private point}: an element belonging to that member and to
no other member of the cover.
For distinct $A_1,\ldots,A_q\in\F$,  let 
$A_i^c=V\setminus A_i$. De Morgan's laws give
$
 \bigcap_{i=1}^q A_i=\varnothing
~\hbox{if and only if}~
 \bigcup_{i=1}^q A_i^c=V.
$
When these conditions hold, for every $i$ we also have
\[
A_i^c\setminus\bigcup_{j\ne i}A_j^c
   =\bigcap_{j\ne i}A_j.
\]
Thus the remaining simplex conditions say exactly that every
$A_i^c$ has a private point. Complementation therefore gives an
exact correspondence between $(q-1)$-simplices in $\F$ and
irredundant $q$-member covers of $V$ in $\G$.

Irredundant covers on disjoint blocks can be combined, since their
private points remain private. To formulate the local assertions
that will be combined, for $v\ge r\ge2$ and $q\ge1$ let
$\cov_r(v,q)$ denote the following \emph{strict} assertion:
every $r$-uniform family on a $v$-element set with more than
$\binom{v-1}{r}$ members contains an irredundant cover of that
set with exactly $q$ members. 
The next lemma shows that strict local assertions already
suffice for a global result with equality.  

\begin{lemma}\label{lem:composition}
Let $r\ge2$, $b\ge2$, $v_i\ge r$, and $q_i\ge1$ for
$1\le i\le b$. Suppose that every $\cov_r(v_i,q_i)$ holds.
Set $v=\sum_i v_i$ and $q=\sum_i q_i$, and assume $v>2r$.
If $\G\subseteq\binom{V}{r}$ has at least
$\binom{v-1}{r}$ members, then it has an irredundant
$q$-member cover unless
\begin{equation}\label{eq:isolatedcomplete}
 \G=\binom{V\setminus\{\alpha\}}{r}
 \quad\text{for some }\alpha\in V.
\end{equation}
\end{lemma}
\begin{proof}
Let $\M=\bin{V}r\setminus\G$ and $m=|\M|$,  
so $m\le \bin{v-1}{r-1}$.
Take a uniformly random ordered partition $(V_1,\ldots,V_b)$
of $V$ with $|V_i|=v_i$. Set
\[
 X_i=\left|\M\cap\bin{V_i}r\right|,
 \qquad d_i=\bin{v_i-1}{r-1}>0,
 \qquad S=\sum_{i=1}^b\frac{X_i}{d_i}.
\]
If $X_i<d_i$, then we have
\[
 \left|\G\cap\bin{V_i}r\right|
 =\bin{v_i}r-X_i
 >\bin{v_i-1}r.
\]
Consequently, if every $X_i<d_i$, each assertion $\cov_r(v_i,q_i)$ yields a local irredundant cover 
$\mathcal C_i \subseteq \mathcal G \cap \bin{V_i}r$ of size $q_i$. Since the blocks $V_i$ partition $V$, the families $\mathcal C_i$ are pairwise vertex-disjoint. Their union $\bigcup_i \mathcal C_i$ therefore constitutes an irredundant $q$-edge cover of $V$, as both edge distinctness and local private points are strictly preserved.

Suppose no such global cover exists. Every permitted partition then
has some $X_i\ge d_i$, so $S\ge1$. For any fixed $A\in\binom{V}r$, the marginal distribution
of $V_i$ is uniform on $\binom{V}{v_i}$, then
\[
 \Pr(A\subseteq V_i)
 =\frac{\binom{v-r}{v_i-r}}{\binom v{v_i}}
 =\frac{\binom{v_i}r}{\binom vr}.
\]
Summing this probability over the $m$ missing edges yields
\begin{align*}
 \mathbb E S
 &=\sum_i\frac{m\bin{v_i}r}{d_i\bin vr}
 =\frac{m}{\bin vr}\sum_i\frac{v_i}{r}
 =\frac{vm}{r\bin vr}
 =\frac{m}{\bin{v-1}{r-1}}\le1.
\end{align*}
Since $S\ge1$ for every permitted partition and
$\mathbb E S=m/\bin{v-1}{r-1}\le1$, we have $m=\bin{v-1}{r-1}$ and
$\mathbb E(S-1)=0$. The random variable $S-1$ is nonnegative on a
finite probability space in which every permitted partition has
positive probability. Hence $S=1$ for every such partition.
Some block has $X_i\ge d_i$, so its contribution must equal one
and every other contribution must vanish. 
Thus, for every permitted partition, exactly one block satisfies $X_i = d_i$, while $X_j = 0$ for all $j \ne i$.

Suppose for contradiction that there exist $E, F \in \M$ such that $E \cap F = \varnothing$. We can construct a partition $(V_1, \ldots, V_b)$ by embedding $E \subseteq V_1$ and $F \subseteq V_2$, and arbitrarily distributing the remaining $v-2r$ vertices of $V\setminus (E \cup F)$ to satisfy $|V_i|=v_i$. This assignment is valid since $v_1, v_2 \ge r$ and $\sum_i v_i = v$. However, this partition yields $E \in \M \cap \bin{V_1}r$ and $F \in \M \cap \bin{V_2}r$, implying $X_1 \ge 1$ and $X_2 \ge 1$. This contradicts the established property that $X_\ell > 0$ for exactly one index $\ell$. Therefore $\M$ is intersecting. Since
$|\M|=m$ and $v>2r$, the equality case of the
EKR theorem (see Lemma \ref{lem:ekr}) gives
\[
 \M=\{A\in\bin{V}r:\alpha\in A\}
\]
for some $\alpha$. This is equivalent to \eqref{eq:isolatedcomplete}.
\end{proof}

The next step is to supply the explicit local covers. We build these systematically: size-$r$ blocks trivially provide single-set covers, while the existence of larger covers is forced by the strict intersection bounds of Lemmas \ref{lem:ekr} and  \ref{lem:chvatal} applied to the complementary families.

\begin{lemma}\label{lem:blocks}
For every $r\ge2$, the following assertions hold:
\[
 \cov_r(r,1),\qquad
 \cov_r(v,2)\quad(r+1\le v\le2r),\qquad
 \cov_r(r+q,q)\quad(q\ge2).
\]
\end{lemma}
\begin{proof}
The first assertion, $\cov_r(r,1)$, is trivial: any family $\F \subseteq \binom{V}{r}$ with $|\F| > \binom{r-1}{r} = 0$ must contain the ground set $V$, which immediately constitutes a $1$-member cover.

To establish $\cov_r(v,2)$ for $r+1 \le v \le 2r$, consider a family $\F \subseteq \binom{V}{r}$ satisfying $|\F| > \binom{v-1}{r}$. We pass to the complementary family $\F^c = \{V \setminus E : E \in \F\}$, whose members have uniformity $u = v - r$. The upper bound $v \le 2r$ implies $2u = 2(v-r) \le v$. Thus, the ground set is large enough to apply the EKR theorem (Lemma~\ref{lem:ekr}). Since $|\F^c| = |\F| > \binom{v-1}{r} = \binom{v-1}{u-1}$, the family $\F^c$ exceeds the maximum size for an intersecting family, guaranteeing the existence of two disjoint members $A_1, A_2 \in \F^c$. By De Morgan's laws, their corresponding original sets $E_1, E_2 \in \F$ cover $V$. Since $|E_1| = |E_2| = r$ and $V = E_1 \cup E_2$ has size $v > r$, neither set can cover $V$ alone. Consequently, both $E_1$ and $E_2$ possess private points, forming the required irredundant $2$-member cover.

The third assertion, $\cov_r(r+q,q)$, splits into two cases based on $q$. When $q=2$, the parameter $v = r+2$ satisfies $r+1 \le v \le 2r$ (as $r \ge 2$), so the result follows directly from the preceding $\cov_r(v,2)$ case.

For the final case, assume $q \ge 3$ and let $v = r+q$. Suppose $\F \subseteq \binom{V}{r}$ satisfies $|\F| > \binom{r+q-1}{r}$. Again taking complements, we obtain $\F^c \subseteq \binom{V}{q}$ with $|\F^c| > \binom{r+q-1}{r} = \binom{r+q-1}{q-1}$. Because $r \ge 2$, the ground set size is $v = q+r \ge q+2$. This allows us to invoke Lemma~\ref{lem:chvatal} with uniformity $q$ and simplex dimension $q-1$. Because $|\F^c|$ strictly exceeds the extremal threshold $\binom{v-1}{q-1}$, the family $\F^c$ must contain a $(q-1)$-simplex. By definition, this simplex consists of $q$ sets in $\F^c$ with an empty total intersection, yet every proper subfamily has a non-empty intersection. By De Morgan's laws, the complements of these $q$ sets in $\F$ cover $V$, and omitting any single set from this collection leaves an element of $V$ uncovered. This exact correspondence yields the necessary irredundant $q$-member cover.
\end{proof}

Combining one-member blocks with a single two-member block
gives the strict high-density range, including its equality
case. The numerical bound in this range is the one  obtained in \cite[p.~170]{Frankl1981}.

\begin{corollary}\label{cor:high}
Let $r\ge2$ and $q\ge3$, and suppose
$
 (q-1)r<v\le qr.
$
Every $\G\subseteq\binom{V}{r}$ with at least
$\binom{v-1}{r}$ members contains an irredundant $q$-member
cover unless it has the form \eqref{eq:isolatedcomplete}.
\end{corollary}
\begin{proof}
We apply Lemma \ref{lem:composition} by explicitly defining a set of block parameters $(v_i, q_i)$. Let the total number of blocks be $b = q-1$. Since $q \ge 3$, we have at least two blocks ($b \ge 2$). We allocate the parameters as follows: configure $q-2$ blocks with $(v_i, q_i) = (r, 1)$, and assign the remaining vertices and cover requirements to a single final block with parameters $ (v - (q-2)r, 2)$.
We first verify the admissibility of this final block's size. The hypothesis states that $(q-1)r < v \le qr$. Subtracting $(q-2)r$ from all terms in this inequality yields
$
 r < v - (q-2)r \le 2r.
$
Because $v$ and $r$ are integers, the strict lower bound implies $r+1 \le v_{q-1} =v - (q-2)r\le 2r$. 

We now confirm that all prerequisites for Lemma~\ref{lem:composition} are satisfied. The block sizes sum exactly to the total number of vertices, as $\sum_{i=1}^{q-1} v_i = (q-2)r + v - (q-2)r = v$. Similarly, the local covers sum to the required global target $q$. 
For the local assertions, Lemma~\ref{lem:blocks} directly guarantees $\cov_r(r,1)$ for the first $q-2$ blocks. Furthermore, because $r+1 \le v_{q-1} \le 2r$, Lemma~\ref{lem:blocks} also guarantees the assertion $\cov_r(v_{q-1},2)$ for the final block. 

Lastly, the global vertex condition requires $v > 2r$. The hypothesis gives $v > (q-1)r$, and since $q \ge 3$, it naturally follows that $v > 2r$. Having satisfied all conditions, Lemma~\ref{lem:composition} applies, demonstrating that $\G$ must contain an irredundant $q$-member cover unless it takes the form in \eqref{eq:isolatedcomplete}.
\end{proof}

To move below the range of Corollary~\ref{cor:high}, we allow one
block to carry more than two members of the cover. The useful
quantity is the difference between the number of vertices and
the required number of cover members. A block with parameters
$(r,1)$ contributes $r-1$ to this difference, while a block with
parameters $(r+q,q)$ contributes $r$, independently of $q$.
A variable two-member block accounts for the remaining residue.
The following integer decomposition makes this precise.

\begin{lemma}\label{lem:middle}
Let $r\ge2$ and $q\ge3$, and suppose
$
 2r-1\le v-q\le(r-1)q-r.
$
Every $r$-uniform family on $V$ with at least
$\binom{v-1}{r}$ members contains an irredundant $q$-member
cover unless it has the form \eqref{eq:isolatedcomplete}.
\end{lemma}
\begin{proof}
Let $s=v-q$. By Euclidean division, write
\[
 s-(2r-1)=h(r-1)+t,\qquad h\ge0,\qquad 0\le t\le r-2.
\]
The upper bound $v-q\le(r-1)q-r$ gives
$
 h(r-1)\le s-(2r-1)\le(r-1)(q-3)-2.
$
Since $r-1>0$, we obtain $h<q-3$, hence $h\le q-4$.
Thus $p=q-h-2\ge2$.
Use $h$ blocks with parameters $(v_i,q_i)=(r,1)$,
one block with $(v_i,q_i)=(r+1+t,2)$, and one block with
$(v_i,q_i)=(r+p,p)$. When $h=0$, the first collection is empty;
there are still two blocks. The cover sizes and the block sizes
sum, respectively, to
$
 h+2+p=q
$
and
\[
 hr+(r+1+t)+(r+p)
   =h(r-1)+2r-1+t+q=s+q=v.
\]
Moreover, $r+1\le r+1+t\le2r-1$, 
ensuring that  every required local
assertion holds by Lemma~\ref{lem:blocks}. Finally,
$v=q+s\ge q+2r-1>2r$.  Lemma~\ref{lem:composition} applies.
\end{proof}

With $r=v-k$ and $q=d+1$, the two bounds in
Lemma~\ref{lem:middle} become $v\le2k-d$ and $v\le dr$,
respectively. Under the hypotheses of Theorem~\ref{thm:main},
Corollary~\ref{cor:high} covers $v>dr$. It therefore remains
to consider
\[
2k-d+1\le v\le dr.
\]
We will treat this range by working directly with $\F$
and its extension degrees.
The local argument will force every member of an extremal
family to contain exactly one uniquely extendible
$(k-1)$-subset. The following rigidity statement is a special
case of Bollob\'as's theorem as recorded in
\cite[Lemma~2.1]{FL}, which requires only at least one such
subset in each member. We include a direct proof of the
special case used here. It first shows that every
$(k-1)$-subset has degree either $1$ or $v-k+1$, then uses
single-element exchanges to identify a common star centre.

\begin{lemma}\label{lem:rigidity}
Let $2\le k<v$ and let $\F\subseteq\binom{V}{k}$ satisfy
$|\F|=\binom{v-1}{k-1}$.
If every member of $\F$ contains exactly one uniquely extendible
$(k-1)$-subset, then $\F$ is a full star.
\end{lemma}
\begin{proof}
Let $r=v-k$, and let
\[
 \Nn=\binom{V}{k}\setminus\F,\qquad
 \Tt_1=\{T\in\binom{V}{k-1}:\deg_\F(T)=1\}.
\]
By hypothesis, each $A\in\F$ contains exactly one member of
$\Tt_1$, and each member of $\Tt_1$ is contained in exactly one
$A\in\F$. Hence
\[
 |\Tt_1|=|\F|=\binom{v-1}{k-1}.
\]
Count pairs $(T,E)$ with $T\in\binom{V}{k-1}$,
$E\in\Nn$, and $T\subseteq E$. Every $E\in\Nn$ contains
exactly $k$ such subsets $T$. Conversely, each $T$ has $r+1$
possible extensions to a $k$-set, of which exactly
$\deg_\F(T)$ belong to $\F$. Thus
\[
 k|\Nn|=\sum_{T\in\binom{V}{k-1}}
                  (r+1-\deg_\F(T)).
\]
We have $|\Nn|=\binom{v-1}{k}$, while the
members of $\Tt_1$ contribute $r\binom{v-1}{k-1}$ to the sum.
Consequently,
\[
 \sum_{T\in\binom{V}{k-1}\setminus\Tt_1}
       (r+1-\deg_\F(T))
   =k\binom{v-1}{k}-r\binom{v-1}{k-1}=0.
\]
All summands are nonnegative, so each is zero. We obtain the
degree dichotomy
\begin{equation}\label{eq:dichotomy}
 \deg_\F(T)\in\{1,r+1\}
 \quad\text{for every }T\in\binom{V}{k-1}.
\end{equation}
Choose $A\in\F$ and let $\alpha\in A$ be the unique element
such that $A\setminus\{\alpha\}\in\Tt_1$.
Fix a target set $A^*\in\binom{V}{k}$ with $\alpha\in A^*$.
If $A\ne A^*$, choose $\beta\in A\setminus A^*$ and
$\gamma\in A^*\setminus A$. In particular,
$\beta\ne\alpha$ and $\gamma\ne\alpha$.
The set $A\setminus\{\beta\}$ is not the uniquely extendible
subset of $A$, so \eqref{eq:dichotomy} gives
$\deg_\F(A\setminus\{\beta\})=r+1$.
Every extension of this subset belongs to $\F$, and hence
$
 B=(A\setminus\{\beta\})\cup\{\gamma\}\in\F.
$
On the other hand,
$E=(A\setminus\{\alpha\})\cup\{\gamma\}\notin\F$,
because $A$ is the unique extension of
$A\setminus\{\alpha\}$ in $\F$.
Both $B$ and $E$ extend
\[
 B\setminus\{\alpha\}
   =(A\setminus\{\alpha,\beta\})\cup\{\gamma\}.
\]
Therefore $1\le\deg_\F(B\setminus\{\alpha\})\le r$,
and \eqref{eq:dichotomy} forces
$\deg_\F(B\setminus\{\alpha\})=1$.
Thus the same element $\alpha$ specifies the uniquely extendible
subset of $B$.

The exchange preserves this property and increases the intersection
with $A^*$ by one. Replacing $A$ by $B$ and repeating, we
therefore reach $A^*$, proving that $A^*\in\F$.
Every $k$-set containing $\alpha$ belongs to $\F$.
Since $|\F|=\binom{v-1}{k-1}=|\mathcal S_\alpha|$, it follows
that $\F=\mathcal S_\alpha$,   the full star with centre $\alpha$.
\end{proof}

We now use links to combine the local extension estimates.
When $v\ge2k-d$, two $k$-sets may intersect in exactly
$d$ elements; deleting their common intersection gives
a disjoint pair in the corresponding link.
The contrapositive of Lemma~\ref{lem:cluster-criterion}
forces both members of every such pair into the marked
subfamily, so Lemma~\ref{lem:marked-intersection} bounds
the size of each link. Summing these bounds over all
$d$-subsets of $V$ and applying the extension-degree
estimates yields a bound for $|\F|$. The strict coefficient
condition in the next lemma allows us to identify the
family when equality holds.

\begin{lemma}\label{lem:local}
Let $d\ge3$, $d<k<v$, and $r=v-k$. Suppose
\begin{equation}\label{eq:localconditions}
v\ge2k-d,\qquad d(r-1)>2r.
\end{equation}
Every $d$-simplex-free family $\F\subseteq\binom{V}{k}$
satisfies $ |\F|\le\binom{v-1}{k-1}$, with equality if and only if it is
a full star.
\end{lemma}
\begin{proof}
Assume $|\F|\ge\binom{v-1}{k-1}$; it suffices to prove that
$\F$ is a full star. The strict inequality in
\eqref{eq:localconditions} forces $r\ge2$ and gives
$
 \frac{r-1}{2r}>\frac1d.
$
Use the extension classes from \eqref{eq:extension-classes}.
Since $k<v$ and $|\F|\ge\binom{v-1}{k-1}$,
Lemma~\ref{lem:extension-budget} gives
\begin{equation}\label{eq:localbudget}
 \sum_{A\in\F}\left(p_1(A)+\frac{r-1}{2r}p_2(A)\right)
   \le\binom{v-1}{k-1}.
\end{equation}
We will compare $|\F|$ with this bound by counting marked
members of the links.
For each $D\in\binom{V}{d}$, define
\[
 \Hh_D=\{A\setminus D:A\in\F,\ D\subseteq A\}
\]
and mark the subfamily
\[
 \Hh_D^*=\{A\setminus D\in\Hh_D:
       D\cap P_1(A)\ne\varnothing
       \text{ or }D\subseteq P_2(A)\}.
\]
The member $A$ is recovered from $A\setminus D$ by adjoining
$D$, so the marking is well defined.
Suppose $A\setminus D$ and $B\setminus D$ are disjoint members
of $\Hh_D$. Since $k-d>0$, the sets $A,B$ are distinct and
$A\cap B=D$. If $A\setminus D$ were unmarked, then
\[
 D\in\binom{A\setminus P_1(A)}{d}\setminus\binom{P_2(A)}{d}.
\]
The hypotheses $d+1\le k$ and $v\ge2k-d$ allow us to apply
Lemma~\ref{lem:cluster-criterion}, giving a $d$-simplex-cluster
in $\F$, a contradiction. Thus $A\setminus D$ is marked.
Interchanging $A$ and $B$ proves that $B\setminus D$ is marked
as well.
We may therefore apply Lemma~\ref{lem:marked-intersection} on
$W=V\setminus D$, with $w=|W|=v-d$ and $u=k-d$.
Here $u\ge1$, and the hypothesis $v\ge2k-d$ gives
$w=v-d\ge2(k-d)=2u$. Also $w-u=r$. Thus
\begin{equation}\label{eq:marked}
 |\Hh_D|\le\binom{v-d-1}{k-d-1}
            +\frac{r}{v-d}|\Hh_D^*|.
\end{equation}
This application includes $v=2k-d$; no equality classification
for the marked-intersection inequality is needed.
It remains to bound the total marked contribution.
For a fixed $A\in\F$, the two marking conditions are mutually
exclusive because $P_1(A)\cap P_2(A)=\varnothing$.
The number of $d$-subsets meeting $P_1(A)$ is
$\binom{k}{d}-\binom{k-p_1(A)}{d}$, and the number contained in
$P_2(A)$ is $\binom{p_2(A)}{d}$.
Double counting, followed by Lemma~\ref{lem:marked-subsets},
the strict coefficient comparison, and \eqref{eq:localbudget},
gives
\begin{align}
 \sum_{D\in\binom{V}{d}}|\Hh_D^*|
 &=\sum_{A\in\F}
     \left[\binom{k}{d}-\binom{k-p_1(A)}{d}
                        +\binom{p_2(A)}{d}\right]\notag\\
 &\le\binom{k-1}{d-1}
       \sum_{A\in\F}\left(p_1(A)+\frac{p_2(A)}{d}\right)\notag\\
 &\le\binom{k-1}{d-1}
       \sum_{A\in\F}\left(p_1(A)+\frac{r-1}{2r}p_2(A)\right)\notag\\
 &\le\binom{k-1}{d-1}\binom{v-1}{k-1}.
 \label{eq:localsubsets}
\end{align}
The use of Lemma~\ref{lem:marked-subsets} is valid because
$d+1\le k<v$.
Each member of $\F$ occurs in exactly $\binom{k}{d}$ links.
Summing \eqref{eq:marked} over $D$ and using
\eqref{eq:localsubsets}, we obtain
\begin{align}
 |\F|\binom{k}{d}
 &\le\binom{v}{d}\binom{v-d-1}{k-d-1}
       +\frac{r}{v-d}\sum_{D\in\binom{V}{d}}|\Hh_D^*|\notag\\
 &\le\binom{v}{d}\binom{v-d-1}{k-d-1}
       +\frac{r}{v-d}\binom{k-1}{d-1}\binom{v-1}{k-1}\notag\\
 &=\binom{v-1}{k-1}\binom{k}{d}.
 \label{eq:completechain}
\end{align}
For the last identity, the relevant ratios are
\[
 \frac{\binom{v}{d}\binom{v-d-1}{k-d-1}}
      {\binom{v-1}{k-1}\binom{k}{d}}
   =\frac{v(k-d)}{k(v-d)},\qquad
 \frac{\binom{k-1}{d-1}}{\binom{k}{d}}=\frac{d}{k}.
\]
Since $r=v-k$, their weighted sum is
$
 \frac{v(k-d)+dr}{k(v-d)}=1.
$
It follows that $|\F|\le\binom{v-1}{k-1}$.
Our initial assumption forces equality. Because $r/(v-d)>0$,
all inequalities in \eqref{eq:localsubsets} and
\eqref{eq:completechain} must be equalities.

We now extract the local information contained in this saturation.
The difference between the two coefficient choices in
\eqref{eq:localsubsets} is
\[
 \binom{k-1}{d-1}
 \left(\frac{r-1}{2r}-\frac1d\right)
 \sum_{A\in\F}p_2(A).
\]
The coefficient is positive, so $p_2(A)=0$ for every $A\in\F$.
With these terms zero, the estimate for each remaining summand is
\begin{equation}\label{eq:pointwise-marking}
 \binom{k}{d}-\binom{k-p_1(A)}{d}
   =\sum_{j=1}^{p_1(A)}\binom{k-j}{d-1}
   \le p_1(A)\binom{k-1}{d-1}.
\end{equation}
Equality holds for $p_1(A)=0$ or $1$. If $p_1(A)\ge2$, the
second term alone falls short by
\[
 \binom{k-1}{d-1}-\binom{k-2}{d-1}
   =\binom{k-2}{d-2}>0.
\]
Summing \eqref{eq:pointwise-marking} gives precisely the
marked-subset inequality in \eqref{eq:localsubsets} after
$p_2(A)=0$. Since that inequality is sharp and every pointwise
deficit is nonnegative, equality holds for each $A$.
Thus $p_1(A)\le1$ for every member. Finally, equality in
\eqref{eq:localbudget} gives
\[
 \sum_{A\in\F}p_1(A)=\binom{v-1}{k-1}=|\F|,
\]
giving $p_1(A)=1$ for every $A\in\F$.
Lemma~\ref{lem:rigidity} now implies that $\F$ is a full star.
Conversely, a full star has $\binom{v-1}{k-1}$ members and
contains no simplex.
\end{proof}

We now assemble the proof of the main result. 
After treating $d=2$, $k=d+1$, and the high-density range,
we separate the remaining parameters at $v=2k-d$.
In the range assigned to the local argument, its strict
coefficient condition fails only at $(d,k,v)=(3,5,8)$.
Two two-member blocks handle this case.

\begin{proof}[Proof of Theorem~\ref{thm:main}]
Assume $|\F|\ge\binom{v-1}{k-1}$. It suffices to prove that
$\F$ is a full star. The case $d=2$ follows from
Lemma~\ref{lem:triangle}. If $k=d+1$, admissibility gives
\[
 v\ge\left\lceil\frac{k^2}{k-1}\right\rceil=k+2,
\]
so Lemma~\ref{lem:chvatal} applies, including its equality clause.
We may therefore assume
$
 d\ge3
$ and 
$k\ge d+2.
$
Let  $r=v-k$, $q=d+1$, and
$\F^c=\{V\setminus A:A\in\F\}$. Then $r\ge2$, $v\le qr$,
and
\[
 |\F^c|=|\F|\ge\binom{v-1}{k-1}=\binom{v-1}{r}.
\]
Since $\F$ is $d$-simplex-free, $\F^c$ has no irredundant
$q$-member cover of $V$. Whenever a cover theorem gives the
exceptional family \eqref{eq:isolatedcomplete}, complementation
therefore yields $\F=\mathcal S_\alpha$.

\medskip\noindent
\textbf{Case 1}: $v>dr$.
Here $(q-1)r<v\le qr$, so Corollary~\ref{cor:high} gives
\eqref{eq:isolatedcomplete} and hence the conclusion.

\medskip\noindent
\textbf{Case 2}: $v\le dr$ and $v\le2k-d$.
Using $k=v-r$ and $q=d+1$, the second inequality is equivalent
to $v\ge d+2r$, or $v-q\ge2r-1$.
The first inequality is equivalent to
$v-q\le(r-1)q-r$. Thus Lemma~\ref{lem:middle} applies and
gives the conclusion.

\medskip\noindent
\textbf{Case 3}: $v\le dr$ and $v\ge2k-d+1$.
The assumption $k\ge d+2$ gives $v\ge d+r+2$.
The inequality $v\ge2k-d+1$, after substituting $k=v-r$,
gives $v\le d+2r-1$. Consequently,
$
 d+r+2\le v\le d+2r-1,
$
which forces $r\ge3$. For $d,r\ge3$,
$
 d(r-1)-2r=(d-2)(r-1)-2\ge0,
$
with equality only when $d=r=3$.
If $(d,r)\ne(3,3)$, then $d(r-1)>2r$ and $v\ge2k-d$,
and  Lemma~\ref{lem:local} gives the conclusion.
If $d=r=3$, the displayed interval forces $v=8$, whence $k=5$
and $q=4$. Two blocks with parameters $(v_i,q_i)=(4,2)$ have
total size $8$ and total cover size $4$.
The local assertion $\cov_3(4,2)$ holds by
Lemma~\ref{lem:blocks}, and $8>2\cdot3$.
Lemma~\ref{lem:composition} therefore gives
\eqref{eq:isolatedcomplete}, as required.

\medskip
The three cases exhaust the remaining admissible integer
parameters. Every simplex-free family of size at least
$\binom{v-1}{k-1}$ is therefore a full star. Since a full star
has exactly this size and contains no simplex, both the upper
bound and the equality characterization follow.
\end{proof}

\end{document}